\documentclass[11pt]{amsart}

\usepackage{amssymb}
\usepackage{amsmath,amssymb,epsf,wrapfig,multicol,epic,eepic,trig,mathrsfs,dsfont,color,graphicx}
 \usepackage[abs]{overpic}
 \usepackage{here}

\theoremstyle{plain}
\newtheorem{thm}{Theorem}
\newtheorem{cor}[thm]{Corollary}
\newtheorem{lem}[thm]{Lemma}
\newtheorem{prop}[thm]{Proposition}

\newtheorem{definition}[thm]{Definition}
\newtheorem{rem}[thm]{Remark}
\newtheorem{example}[thm]{Example}

\def\hsymbu#1{\smash{\lower1.7ex\hbox{\huge$#1$}}}

\def\rmoveio#1#2{
\setlength{\unitlength}{#1}
\begin{picture}(50,30)
\put(5,0){\line(0,1){30}}

{\allinethickness{.8pt}
\put(10,15){\vector(1,0){13}}
\put(23,15){\vector(-1,0){13}}}

\qbezier(25,0)(25,20)(40,20)
\qbezier(40,20)(45,20)(45,15)
\qbezier(45,15)(45,10)(40,10)
\qbezier(40,10)(35,10)(31,14)
\qbezier(28,17)(25,25)(25,30)

\ifnum#2=2
\put(30,23){\makebox{${\Huge c_{1}}$}}
\fi

\end{picture}
}

\def\rmoveiio#1#2{
\setlength{\unitlength}{#1}
\begin{picture}(60,30)
\put(5,0){\line(0,1){30}}
\put(15,0){\line(0,1){30}}

{\allinethickness{.8pt}
\put(20,15){\vector(1,0){15}}
\put(35,15){\vector(-1,0){15}}}

\qbezier(40,0)(42,1)(47,3)
\qbezier(52,6)(68,15)(52,24)
\qbezier(47,27)(42,30)(40,30)

\qbezier(60,0)(20,15)(60,30)

\ifnum#2=2
\put(47,17){\makebox{${\Huge c_{1}}$}}
\put(47,9){\makebox{${\Huge c_{2}}$}}
\fi

\ifnum#2=3
\put(2,2){\path(0,0)(3,-3)(6,0)}
\put(12,27){\path(0,0)(3,3)(6,0)}
\put(40,3){\path(0,0)(0,-3)(3,-3)}
\put(60,27){\path(0,0)(0,3)(-3,3)}
\put(47,16){\makebox{${\Huge c_{1}}$}}
\put(47,10){\makebox{${\Huge c_{2}}$}}
\fi

\end{picture}
}
\def\rmoveiiio#1#2{
\setlength{\unitlength}{#1}
\begin{picture}(75,30)
\put(0,0){\line(1,1){15}}
\qbezier(15,15)(20,20)(20,30)

\put(10,0){\line(-1,1){4}}
\qbezier(4,6)(-5,15)(5,25)
\put(5,25){\line(1,1){5}}

\qbezier(20,0)(20,10)(16,14)
\put(14,16){\line(-1,1){8}}
\put(4,26){\line(-1,1){4}}

{\allinethickness{.8pt}
\put(29,15){\vector(1,0){15}}
\put(44,15){\vector(-1,0){15}}}

\qbezier(50,0)(50,10)(55,15)
\put(55,15){\line(1,1){15}}

\put(60,0){\line(1,1){5}}
\qbezier(65,5)(75,15)(66,24)
\put(64,26){\line(-1,1){4}}

\put(70,0){\line(-1,1){4}}
\put(64,6){\line(-1,1){8}}
\qbezier(54,16)(50,20)(50,30)

\ifnum#2=2

\put(10,23){\makebox{${\Huge c_{1}}$}}
\put(10,3){\makebox{${\Huge c_{2}}$}}
\put(20,13){\makebox{${\Huge c_{3}}$}}

\put(70,23){\makebox{${\Huge c'_{2}}$}}
\put(70,3){\makebox{${\Huge c'_{1}}$}}
\put(60,13){\makebox{${\Huge c'_{3}}$}}
\fi

\end{picture}
}
\def\rmovevio#1#2{
\setlength{\unitlength}{#1}
\begin{picture}(50,30)
\put(5,0){\line(0,1){30}}

{\allinethickness{.8pt}
\put(10,15){\vector(1,0){15}}
\put(25,15){\vector(-1,0){15}}}

\qbezier(30,0)(30,20)(45,20)
\qbezier(45,20)(50,20)(50,15)
\qbezier(50,15)(50,10)(45,10)
\qbezier(45,10)(30,10)(30,30)

\put(34,15){\circle{5}}

\ifnum#2=2
\put(3,28){\path(0,0)(2,2)(4,0)}
\put(28,28){\path(0,0)(2,2)(4,0)}
\put(38,15){\makebox{${\Huge c_{1}}$}}
\fi

\end{picture}
}

\def\rmoveviio#1#2{
\setlength{\unitlength}{#1}
\begin{picture}(60,30)
\put(5,0){\line(0,1){30}}
\put(15,0){\line(0,1){30}}

{\allinethickness{.8pt}
\put(20,15){\vector(1,0){15}}
\put(35,15){\vector(-1,0){15}}}

\qbezier(40,0)(80,15)(40,30)
\qbezier(60,0)(20,15)(60,30)
\put(50,4){\circle{5}}
\put(50,25){\circle{5}}

\ifnum#2=2
\put(2,27){\path(0,0)(3,3)(6,0)}
\put(12,27){\path(0,0)(3,3)(6,0)}
\put(40,27){\path(0,0)(0,3)(3,3)}
\put(60,27){\path(0,0)(0,3)(-3,3)}
\put(47,15){\makebox{${\Huge c_{1}}$}}
\put(47,10){\makebox{${\Huge c_{2}}$}}
\fi

\ifnum#2=3
\put(2,2){\path(0,0)(3,-3)(6,0)}
\put(12,27){\path(0,0)(3,3)(6,0)}
\put(40,3){\path(0,0)(0,-3)(3,-3)}
\put(60,27){\path(0,0)(0,3)(-3,3)}
\put(47,15){\makebox{${\Huge c_{1}}$}}
\put(47,10){\makebox{${\Huge c_{2}}$}}
\fi

\end{picture}
}

\def\rmoveviiio#1#2{
\setlength{\unitlength}{#1}
\begin{picture}(70,30)
\put(0,0){\line(1,1){15}}
\qbezier(15,15)(20,20)(20,30)

\put(10,0){\line(-1,1){5}}
\qbezier(5,5)(-5,15)(5,25)
\put(5,25){\line(1,1){5}}

\qbezier(20,0)(20,10)(15,15)
\put(15,15){\line(-1,1){15}}

\put(5,5){\circle{5}}
\put(15,15){\circle{5}}
\put(5,25){\circle{5}}

{\allinethickness{.8pt}
\put(28,15){\vector(1,0){14}}
\put(42,15){\vector(-1,0){14}}}

\qbezier(50,0)(50,10)(55,15)
\put(55,15){\line(1,1){15}}

\put(60,0){\line(1,1){5}}
\qbezier(65,5)(75,15)(65,25)
\put(65,25){\line(-1,1){5}}

\put(70,0){\line(-1,1){15}}
\qbezier(55,15)(50,20)(50,30)

\put(65,5){\circle{5}}
\put(55,15){\circle{5}}
\put(65,25){\circle{5}}

\ifnum#2=2
\put(0,27){\path(0,0)(0,3)(3,3)}
\put(7,30){\path(0,0)(3,0)(3,-3)}
\put(17,27){\path(0,0)(3,3)(6,0)}

\put(20,13){\makebox{${\Huge c_{1}}$}}

\put(47,27){\path(0,0)(3,3)(6,0)}
\put(60,27){\path(0,0)(0,3)(3,3)}
\put(67,30){\path(0,0)(3,0)(3,-3)}

\put(60,13){\makebox{${\Huge c'_{1}}$}}
\fi

\end{picture}
}

\def\rmovevivo#1#2{
\setlength{\unitlength}{#1}
\begin{picture}(70,30)
\put(0,0){\line(1,1){15}}
\qbezier(15,15)(20,20)(20,30)

\put(10,0){\line(-1,1){4}}
\qbezier(4,6)(-5,15)(5,25)
\put(5,25){\line(1,1){5}}

\qbezier(20,0)(20,10)(16,14)
\put(16,14){\line(-1,1){16}}

\put(15,15){\circle{5}}
\put(5,25){\circle{5}}

{\allinethickness{.8pt}
\put(28,15){\vector(1,0){14}}
\put(42,15){\vector(-1,0){14}}}

\qbezier(50,0)(50,10)(55,15)
\put(55,15){\line(1,1){9}}
\put(66,26){\line(1,1){4}}

\put(60,0){\line(1,1){5}}
\qbezier(65,5)(75,15)(65,25)
\put(65,25){\line(-1,1){5}}

\put(70,0){\line(-1,1){16}}
\qbezier(54,16)(50,20)(50,30)

\put(65,5){\circle{5}}
\put(55,15){\circle{5}}

\ifnum#2=2
\put(0,27){\path(0,0)(0,3)(3,3)}
\put(7,30){\path(0,0)(3,0)(3,-3)}
\put(17,27){\path(0,0)(3,3)(6,0)}

\put(20,13){\makebox{${\Huge c_{1}}$}}

\put(47,27){\path(0,0)(3,3)(6,0)}
\put(60,27){\path(0,0)(0,3)(3,3)}
\put(67,30){\path(0,0)(3,0)(3,-3)}

\put(60,13){\makebox{${\Huge c'_{1}}$}}
\fi

\end{picture}
}
\def\rmoveti#1{
\setlength{\unitlength}{#1}
\begin{picture}(60,20)
\put(0,10){\line(1,0){20}}
\put(15,0){\line(0,1){20}}
\put(15,10){\circle{3}}
{
\put(7,7){\line(0,1){6}}} 
{\allinethickness{.8pt}
\put(25,10){\vector(1,0){10}}
\put(35,10){\vector(-1,0){10}}}

\put(40,10){\line(1,0){20}}
\put(45,0){\line(0,1){20}}
\put(45,10){\circle{3}}
{
\put(53,7){\line(0,1){6}}}
\end{picture}
}

\def\rmovetiio#1{
\setlength{\unitlength}{#1}
\begin{picture}(40,20)
\put(5,0){\line(0,1){20}}
{
\put(2.,7){\line(1,0){6}} 
\put(2.,13){\line(1,0){6}} }
{\allinethickness{.8pt}
\put(15,10){\vector(1,0){10}}
\put(25,10){\vector(-1,0){10}}}
\put(35,0){\line(0,1){20}}
\end{picture}
}

\def\rmovetiiio#1#2{
\setlength{\unitlength}{#1}
\begin{picture}(70,20)

\put(0,0){\line(1,1){20}}
\put(20,0){\line(-1,1){9}}
\put(9,11){\line(-1,1){9}}

\put(7,3){\line(-1,1){4}}
\put(13,3){\line(1,1){4}}
\put(3,13){\line(1,1){4}}
\put(17,13){\line(-1,1){4}}

{\allinethickness{.8pt}
\put(25,10){\vector(1,0){10}}
\put(35,10){\vector(-1,0){10}}}

\qbezier(40,5)(48,20)(54,11)
\qbezier(56, 9)(62,0)(70,15)

\qbezier(40,15)(48,0)(55,10)
\qbezier(55,10)(62,20)(70,5)

\put(43,10){\circle{5}}
\put(67,10){\circle{5}}

\ifnum#2=2

\put(15,8){\makebox{${\Huge c_{1}}$}}

\put(53,0){\makebox{${\Huge c'_{1}}$}}
\fi

\end{picture}
}

\def\rmovevfo#1#2{
\setlength{\unitlength}{#1}
\begin{picture}(70,30)
\put(0,0){\line(1,1){15}}
\qbezier(15,15)(20,20)(20,30)

\put(10,0){\line(-1,1){4}}
\qbezier(4,6)(-5,15)(5,25)
\put(5,25){\line(1,1){5}}

\qbezier(20,0)(20,10)(16,14)
\put(14,16){\line(-1,1){14}}

\put(5,25){\circle{5}}

{\allinethickness{.8pt}
\put(28,15){\vector(1,0){14}}
\put(42,15){\vector(-1,0){14}}}

\qbezier(50,0)(50,10)(55,15)
\put(55,15){\line(1,1){15}}

\put(60,0){\line(1,1){5}}
\qbezier(65,5)(75,15)(66,24)
\put(64,26){\line(-1,1){4}}

\put(70,0){\line(-1,1){14}}
\qbezier(54,16)(50,20)(50,30)

\put(65,5){\circle{5}}

\ifnum#2=2
\put(0,27){\path(0,0)(0,3)(3,3)}
\put(7,30){\path(0,0)(3,0)(3,-3)}
\put(17,27){\path(0,0)(3,3)(6,0)}

\put(20,13){\makebox{${\Huge c_{1}}$}}

\put(47,27){\path(0,0)(3,3)(6,0)}
\put(60,27){\path(0,0)(0,3)(3,3)}
\put(67,30){\path(0,0)(3,0)(3,-3)}

\put(60,13){\makebox{${\Huge c'_{1}}$}}
\fi

\end{picture}
}

\def\alexnum#1{
\setlength{\unitlength}{#1}
\begin{picture}(90,20)

\put(5,0){\line(1,1){20}}
\put(25,0){\line(-1,1){9}}
\put(14,11){\line(-1,1){9}}

\put(5,0){\path(0,3)(0,0)(3,0)}
\put(25,0){\path(-3,0)(0,0)(0,3)}

\put(0,15){\makebox{{$i$}}}
\put(0,0){\makebox{{$i$}}}
\put(27,15){\makebox{{$i+1$}}}
\put(27,0){\makebox{{$i+1$}}}

\put(55,0){\line(1,1){9}}
\put(75,0){\line(-1,1){20}}
\put(66,11){\line(1,1){9}}

\put(55,0){\path(0,3)(0,0)(3,0)}
\put(75,0){\path(-3,0)(0,0)(0,3)}

\put(50,15){\makebox{{$i$}}}
\put(50,0){\makebox{{$i$}}}
\put(77,15){\makebox{{$i+1$}}}
\put(77,0){\makebox{{$i+1$}}}

\end{picture}
}
\def\alexnumv#1{
\setlength{\unitlength}{#1}
\begin{picture}(35,20)

\put(5,0){\line(1,1){20}}
\put(25,0){\line(-1,1){20}}
\put(15,10){\circle{4}}

\put(5,0){\path(0,3)(0,0)(3,0)}
\put(25,0){\path(-3,0)(0,0)(0,3)}

\put(0,15){\makebox{{$j$}}}
\put(0,0){\makebox{{$i$}}}
\put(27,15){\makebox{{$i$}}}
\put(27,0){\makebox{{$j$}}}

\end{picture}
}
\def\orientedcutpt#1{
\setlength{\unitlength}{#1}
\begin{picture}(20,20)
\put(5,0){\line(0,1){20}}
\put(5.,8){\path(0,0)(2,2)(-2,2)(0,0)} 

\end{picture}
}
\def\alexnumcp#1{
\setlength{\unitlength}{#1}
\begin{picture}(20,20)
\put(5,0){\line(0,1){20}}
\put(5.,8){\path(0,0)(2,2)(-2,2)(0,0)} 

\put(7,13){\makebox{$i$}}
\put(7,3){\makebox{$i+1$}}
\end{picture}
}
\def\canocutsysic#1#2{
\setlength{\unitlength}{#1}
\begin{picture}(130,20)

\put(5,0){\line(1,1){20}}
\put(25,0){\line(-1,1){8}}
\put(5,20){\line(1,-1){8}}

\put(20,15){\path(0,0)(3,0)(0,3)(0,0)}
\put(20,5){\path(0,0)(3,0)(0,-3)(0,0)}

\put(5,0){\path(0,3)(0,0)(3,0)}
\put(25,0){\path(-3,0)(0,0)(0,3)}

\put(50,0){\line(1,1){8}}
\put(70,20){\line(-1,-1){8}}
\put(70,0){\line(-1,1){20}}

\put(65,15){\path(0,0)(3,0)(0,3)(0,0)}
\put(65,5){\path(0,0)(3,0)(0,-3)(0,0)}

\put(50,0){\path(0,3)(0,0)(3,0)}
\put(70,0){\path(-3,0)(0,0)(0,3)}

\ifnum#2=2
\put(0,15){\makebox{{\small $0$}}}
\put(0,0){\makebox{{\small $0$}}}
\put(27,15){\makebox{{\small $0$}}}
\put(27,0){\makebox{{\small $0$}}}
\put(15,14){\makebox{{\small $1$}}}
\put(15,1){\makebox{{\small $1$}}}
\put(45,15){\makebox{{\small $0$}}}
\put(45,0){\makebox{{\small $0$}}}
\put(72,15){\makebox{{\small $0$}}}
\put(72,0){\makebox{{\small $0$}}}
\put(60,14){\makebox{{\small $1$}}}
\put(60,1){\makebox{{\small $1$}}}

\fi

\put(95,0){\line(1,1){20}}
\put(115,0){\line(-1,1){20}}
\put(105,10){\circle{4}}

\put(95,0){\path(0,3)(0,0)(3,0)}
\put(115,0){\path(-3,0)(0,0)(0,3)}

\ifnum#2=2
\put(90,15){\makebox{{\small $0$}}}
\put(90,0){\makebox{{\small $0$}}}
\put(117,15){\makebox{{\small $0$}}}
\put(117,0){\makebox{{\small $0$}}}
\fi

\end{picture}
}
\def\ocpmovei#1{
\setlength{\unitlength}{#1}
\begin{picture}(60,20)
\put(0,0){\line(1,1){20}}
\put(20,0){\line(-1,1){20}}
\put(10,10){\circle{5}}
\put(15,15){\path(0,0)(3,0)(0,3)(0,0)}

{\allinethickness{.8pt}
\put(25,10){\vector(1,0){10}}
\put(35,10){\vector(-1,0){10}}}

\put(40,0){\line(1,1){20}}
\put(60,0){\line(-1,1){20}}
\put(50,10){\circle{5}}
\put(43,3){\path(0,0)(3,0)(0,3)(0,0)}
\end{picture}
}

\def\ocpmoveii#1{
\setlength{\unitlength}{#1}
\begin{picture}(50,20)

\put(5,0){\line(0,1){20}}
{\allinethickness{.8pt}

\put(10,10){\vector(1,0){15}}
\put(25,10){\vector(-1,0){15}}}

\put(30,0){\line(0,1){20}}
\put(30.,5){\path(0,0)(2,2)(-2,2)(0,0)} 
\put(30,15){\path(0,0)(2,-2)(-2,-2)(0,0)}

\put(35,7){\makebox{{or}}}

\put(47,0){\line(0,1){20}}
\put(47.,7){\path(0,0)(2,-2)(-2,-2)(0,0)} 
\put(47.,13){\path(0,0)(2,2)(-2,2)(0,0)}

\end{picture}
}

\def\ocpmoveiiia#1#2{
\setlength{\unitlength}{#1}
\begin{picture}(60,20)

\put(0,0){\line(1,1){20}}
\put(20,0){\line(-1,1){9}}
\put(9,11){\line(-1,1){9}}

{\allinethickness{.8pt}
\put(25,10){\vector(1,0){10}}
\put(35,10){\vector(-1,0){10}}}

\put(40,0){\line(1,1){20}}
\put(60,0){\line(-1,1){9}}
\put(49,11){\line(-1,1){9}}

\put(43,3){\path(0,0)(3,0)(0,3)(0,0)}
\put(57,3){\path(0,0)(0,3)(-3,0)(0,0)}
\put(57,17){\path(0,0)(-3,0)(0,-3)(0,0)}
\put(43,17){\path(0,0)(0,-3)(3,0)(0,0)}

\ifnum#2=2
\put(14,9){\makebox{{$c_1$}}}
\put(54,9){\makebox{{$c_1$}}}
\fi

\end{picture}
}
\def\ocpmoveiiib#1#2{
\setlength{\unitlength}{#1}
\begin{picture}(60,20)

\put(0,0){\line(1,1){20}}
\put(20,0){\line(-1,1){9}}
\put(9,11){\line(-1,1){9}}

{\allinethickness{.8pt}
\put(25,10){\vector(1,0){10}}
\put(35,10){\vector(-1,0){10}}}

\put(40,0){\line(1,1){20}}
\put(60,0){\line(-1,1){9}}
\put(49,11){\line(-1,1){9}}

\put(45,5){\path(0,0)(-3,0)(0,-3)(0,0)}
\put(55,5){\path(0,0)(0,-3)(3,0)(0,0)}
\put(55,15){\path(0,0)(3,0)(0,3)(0,0)}
\put(45,15){\path(0,0)(0,3)(-3,0)(0,0)}

\ifnum#2=2
\put(14,9){\makebox{{$c_1$}}}
\put(54,9){\makebox{{$c_1$}}}
\fi

\end{picture}
}
\def\stnocutsysi#1#2{
\setlength{\unitlength}{#1}
\begin{picture}(130,20)

\put(5,0){\line(1,1){20}}
\put(25,0){\line(-1,1){8}}
\put(5,20){\line(1,-1){8}}

\put(5,0){\path(0,3)(0,0)(3,0)}
\put(25,0){\path(-3,0)(0,0)(0,3)}

\put(50,0){\line(1,1){8}}
\put(70,20){\line(-1,-1){8}}
\put(70,0){\line(-1,1){20}}

\put(50,0){\path(0,3)(0,0)(3,0)}
\put(70,0){\path(-3,0)(0,0)(0,3)}

\ifnum#2=2
\put(0,15){\makebox{{\small $i$}}}
\put(0,0){\makebox{{\small $i$}}}
\put(27,15){\makebox{{\small $i+1$}}}
\put(27,0){\makebox{{\small $i+1$}}}
\put(45,15){\makebox{{\small $j$}}}
\put(45,0){\makebox{{\small $j$}}}
\put(72,15){\makebox{{\small $j+1$}}}
\put(72,0){\makebox{{\small $j+1$}}}

\fi

\put(95,0){\line(1,1){20}}
\put(115,0){\line(-1,1){20}}
\put(105,10){\circle{4}}

\put(100,5){\path(0,0)(-3,0)(0,-3)(0,0)}
\put(112,3){\path(0,0)(-3,0)(0,3)(0,0)}

\put(95,0){\path(0,3)(0,0)(3,0)}
\put(115,0){\path(-3,0)(0,0)(0,3)}

\ifnum#2=2
\put(90,15){\makebox{{\small $k$}}}
\put(90,0){\makebox{{$\small k$}}}
\put(117,15){\makebox{{\small $k+1$}}}
\put(117,0){\makebox{{\small $k+1$}}}
\fi

\end{picture}
}

\def\smoothLR#1{
\setlength{\unitlength}{#1}
\begin{picture}(170,50)
  \thicklines
\put(0,15){\line(1,1){20}}
\put(0,35){\line(1,-1){8}}
\put(20,15){\line(-1,1){8}}
\put(0,15){\path(0,5)(0,0)(5,0)}
\put(20,15){\path(-5,0)(0,0)(0,5)}
\put(30,15){\line(1,1){8}}
\put(50,35){\line(-1,-1){8}}
\put(30,35){\line(1,-1){20}}
\put(30,15){\path(0,5)(0,0)(5,0)}
\put(50,15){\path(-5,0)(0,0)(0,5)}

\put(0,5){\makebox{a classical crossing $c$}}
\put(60,25){\vector(1,0){15}}
\qbezier(115,15)(125,25)(115,35)
\qbezier(135,15)(125,25)(135,35)

\qbezier[10](115,15)(105,5)(95,15)
\qbezier[10](115,35)(105,45)(95,35)
\qbezier[10](95,15)(85,25)(95,35)

\qbezier[10](135,15)(145,5)(155,15)
\qbezier[10](135,35)(145,45)(155,35)
\qbezier[10](155,15)(165,25)(155,35)

\put(107,23){\makebox{$D_c^{\rm R}$}}
\put(133,23){\makebox{$D_c^{\rm L}$}}

\put(115,15){\path(0,5)(0,0)(5,0)}
\put(135,15){\path(-5,0)(0,0)(0,5)}

\put(120,0){\makebox{$D_c$}}
\end{picture}
}

\def\smoothOUpositive#1{
\setlength{\unitlength}{#1}
\begin{picture}(170,50)
  \thicklines
\put(30,15){\line(1,1){20}}
\put(30,35){\line(1,-1){8}}
\put(50,15){\line(-1,1){8}}
\put(30,15){\path(0,5)(0,0)(5,0)}
\put(50,15){\path(-5,0)(0,0)(0,5)}

\put(0,5){\makebox{a positive crossing $c$}}

\put(60,25){\vector(1,0){15}}

\qbezier(115,15)(125,25)(115,35)
\qbezier(135,15)(125,25)(135,35)

\qbezier[10](115,15)(105,5)(95,15)
\qbezier[10](115,35)(105,45)(95,35)
\qbezier[10](95,15)(85,25)(95,35)

\qbezier[10](135,15)(145,5)(155,15)
\qbezier[10](135,35)(145,45)(155,35)
\qbezier[10](155,15)(165,25)(155,35)

\put(107,23){\makebox{$D_c^{\rm O}$}}
\put(133,23){\makebox{$D_c^{\rm U}$}}

\put(115,15){\path(0,5)(0,0)(5,0)}
\put(135,15){\path(-5,0)(0,0)(0,5)}

\put(120,0){\makebox{$D_c$}}

\end{picture}
}

\def\smoothOUnegative#1{
\setlength{\unitlength}{#1}
\begin{picture}(170,50)
  \thicklines
\put(30,15){\line(1,1){8}}
\put(50,35){\line(-1,-1){8}}
\put(30,35){\line(1,-1){20}}
\put(30,15){\path(0,5)(0,0)(5,0)}
\put(50,15){\path(-5,0)(0,0)(0,5)}

\put(0,5){\makebox{a negative crossing $c$}}

\put(60,25){\vector(1,0){15}}

\qbezier(115,15)(125,25)(115,35)
\qbezier(135,15)(125,25)(135,35)

\qbezier[10](115,15)(105,5)(95,15)
\qbezier[10](115,35)(105,45)(95,35)
\qbezier[10](95,15)(85,25)(95,35)

\qbezier[10](135,15)(145,5)(155,15)
\qbezier[10](135,35)(145,45)(155,35)
\qbezier[10](155,15)(165,25)(155,35)

\put(107,23){\makebox{$D_c^{\rm U}$}}
\put(133,23){\makebox{$D_c^{\rm O}$}}

\put(115,15){\path(0,5)(0,0)(5,0)}
\put(135,15){\path(-5,0)(0,0)(0,5)}

\put(120,0){\makebox{$D_c$}}

\end{picture}
}

\def\indloopi#1#2{
\setlength{\unitlength}{#1}
\begin{picture}(15,10)
\ifnum#2=1
\put(0,5){\line(1,0){15}}
\put(4,5){\line(0,1){3}}
\put(11,5){\line(0,1){3}}
\fi
\ifnum#2=2
\put(0,5){\line(1,0){15}}
\fi
\end{picture}
}
\def\indloopii#1#2{
\setlength{\unitlength}{#1}
\begin{picture}(15,15)
\ifnum#2=1
\put(5,0){\line(0,1){15}}
\put(0,7){\line(1,0){15}}
\put(10,7){\line(0,1){3}}
\put(2.5,6){\makebox{$\bigcirc$}}
\fi
\ifnum#2=2
\put(8,0){\line(0,1){15}}
\put(3,7){\line(0,1){3}}
\put(0,7){\line(1,0){15}}
\put(5.5,6){\makebox{$\bigcirc$}}
\fi
\end{picture}
}

\def\indloopiii#1#2{
\setlength{\unitlength}{#1}
\begin{picture}(17,10)
\ifnum#2=1
{\linethickness{2pt}
\put(8,2){\line(0,1){6}}}
\put(0,5){\line(1,0){15}}
\put(12,5){\line(0,1){3}}
\fi
\ifnum#2=2
{\linethickness{2pt}
\put(8,2){\line(0,1){6}}}
\put(4,5){\line(0,-1){3}}
\put(0,5){\line(1,0){15}}
\fi
\end{picture}
}

\begin{document}
\title[Invariants of virtual links and twisted links using affine indices]
{Invariants of virtual links and twisted links using affine indices}

\author{Naoko Kamada}

\address{Graduate School of Science,  Nagoya City University\\ 
1 Yamanohata, Mizuho-cho, Mizuho-ku, Nagoya, Aichi 467-8501 Japan
}

\author{Seiichi Kamada}

\address{Department of Mathematics, Osaka University\\
Toyonaka,  Osaka 560-0043, Japan
}

\thanks{This work was supported by JSPS KAKENHI Grant Number 
23K03118 
and 
19H01788. 
}

\date{}

\begin{abstract} 
The affine index polynomial and the $n$-writhe are invariants of virtual knots which are introduced by Kauffman \cite{rkauF} and by 
Satoh and Taniguchi \cite{rST} independently. They are defined by using indices assigned to each classical crossing, which we call affine indices in this paper.  
We discuss a relationship between the invariants and generalize them to invariants of virtual links. The invariants for virtual links can be also computed by using cut systems. We also introduce invariants of twisted links by using affine indices.  
\end{abstract}




\maketitle


\section{Introduction}

Virtual links are a generalization of links introduced by Kauffman~\cite{rkauD}. They are defined as equivalence classes of link diagrams possibly with virtual crossings, called virtual link diagrams, and they correspond to equivalence classes of Gauss diagrams \cite{rkauD}. 
They are also in one-to-one correspondence with 
stable equivalence classes of  links in thickened oriented surfaces \cite{rCKS,rkk}, and with abstract links defined in \cite{rkk}.  

The affine index polynomial $P_K(t)$ and the $n$-writhe $J_n(K)$ $(n\neq 0)$ are invariants of virtual knots which are introduced by Kauffman \cite{rkauF} and by 
Satoh and Taniguchi \cite{rST} independently. They are defined by using indices assigned to each classical crossing, which we call affine indices in this paper.  Affine indices are a refinement of indices introduced by Henrich~\cite{rHen} and Im, Lee and Lee~\cite{rILL}. 
The odd writhe invariant \cite{rkauE}, the odd writhe polynomial \cite{rCheng2014} and the invariants in  \cite{rHen} and \cite{rILL} are recovered from the affine index polynomial and the $n$-writhe.  
The writhe polynomial defined in \cite{rChengGao2013} is also another variation of the affine index polynomial and the $n$-writhe.  
The writhe polynomial for a virtual link is also introduced in Xu's paper~\cite{rXu}. 

One of the purposes of this paper is to discuss a relationship between the affine index polynomial and the $n$-writhe (Proposition~\ref{prop:STandP}) and generalize them to invariants of virtual links. 

For a virtual link diagram $D$, we define the over and under $n$-writhe,  
$J^{\rm O}_n(D)$ and $J^{\rm U}_n(D)$, for $n \neq 0$, and over, under and over-under affine index polynomials, 
$P^{\rm O}_D(t)$, $P^{\rm U}_D(t)$ and  
$P^{\rm OU}_D(s, t)$.  All of these are virtual link invariants (Theorem~\ref{virtuallinkinvariants}).

Kauffman~\cite{rKauffman2018} defined the affine index polynomial $P_L(t)$ for a {\it compatible} virtual link. 
For a compatible virtual link, our invariants $P^{\rm O}_L(t)$ and $P^{\rm U}_L(t)$ are essentially the same with $P_L(t)$ (Remark~\ref{rem:compatible}).

Another purpose is to give an alternative definition to these invariants, or an alternative method of computing these invariants, using cut systems. 
A cut system of a virtual link diagram $D$ is a finite set of oriented cut points such that the diagram with the cut points admits an Alexander numbering.  Cut systems are used for constructing a checkerboard colorable, a mod-$m$ almost classical or almost classical virtual link diagram from a given virtual link diagram (cf. \cite{rkn2017, rkn2018, rkn2019}). Oriented cut points are a generalization of (unoriented) cut points introduced by Dye~\cite{rdye2, rdye2017}.  

Twisted links are also a generalization of links introduced by Bourgoin~\cite{rbor}.  They are defined as equivalence classes of link diagrams possibly with virtual crossings and bars, called twisted link diagrams.  
They are in one-to-one correspondence with 
stable equivalence classes of  links in thickened surfaces, and with abstract links over unoriented surfaces.  
We introduce invariants of twisted links using affine indices via the double covering method introduced in \cite{rkk16}.   
In the last section, we introduce another invariant of twisted links using affine indices.

This paper is organized as follows: 
In Section~\ref{sect:virtualknot} we recall the notion of virtual knots and links. 
In~Section~\ref{sect:defmain} we recall the $n$-writhe and the affine index polynomial for a virtual knot, and observe a relationship between them. 
In~Section~\ref{sect:links} we generalize the invariants to virtual link invariants.
In~Section~\ref{sect:cut} we recall the notions of Alexander numberings, oriented cut points and cut systems. Then 
we discuss an alternative definition of the invariants for virtual knots and links in terms of cut systems.
In Section~\ref{sect:twisted} we recall the definition of twisted knots and links, and construct invariants of twisted links via the double covering method. 
In Section~\ref{sect:newtwistedinv} another invariant of twisted links is defined by using the affine indices modulo $2$. 

We would like to thank the referee for their valuable comments and suggestions, which have significantly improved the quality of this paper. In particular, Remarks~\ref{rem:Xu} and~\ref{rem:linking} were suggested by the referee.

\section{Virtual knots and links}\label{sect:virtualknot} 

A {\it virtual link diagram\/} is a collection of immersed oriented loops in $\mathbb{R}^2$ such that the multiple points are transverse double points which are classified into classical crossings and virtual crossings: 
A {\it classical crossing} is a crossing with over/under information as in usual link diagrams, and a {\it virtual crossing} is a crossing without over/under information \cite{rkauD}.  A virtual crossing is depicted as a crossing encircled with a small circle in oder to distinguish from classical crossings. (Such a circle is not considered as a component of the virtual link diagram.)  A classical crossing is also called a positive or negative crossing according to the sign of the crossing as usual in knot theory.  A {\it virtual knot diagram\/} is a virtual link  diagram with one component. 

{\it Generalized Reidemeister moves} are local moves depicted in Figure~ \ref{fgmoves}: The 3 moves on the top are {\it (classical) Reidemeister moves} and the 4 moves on the bottom are so-called {\it virtual Reidemeister moves}. 
Two virtual link diagrams are said to be {\it equivalent} if they are related by a finite sequence of generalized Reidemeister moves and isotopies of $\mathbb{R}^2$.  
A {\it virtual link\/}   is an equivalence class 
of virtual link diagrams.  A {\it virtual knot\/} is an equivalence class of a virtual knot diagram. 

\vspace{0.3cm}
\begin{figure}[h]
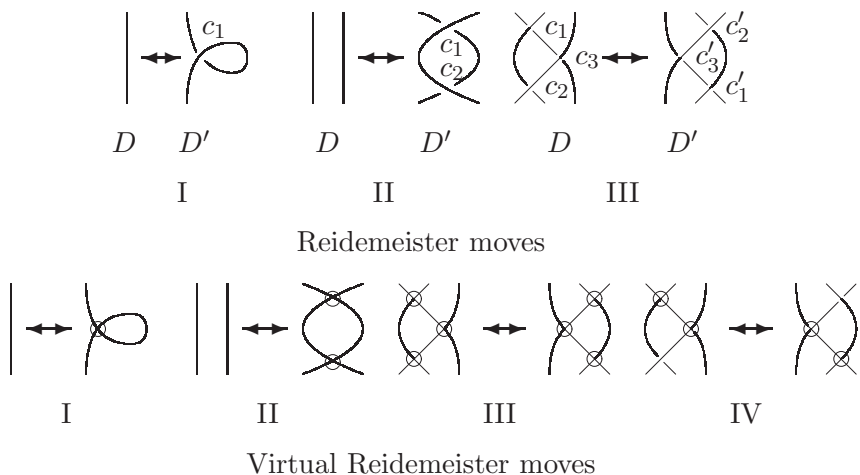

\centerline{
\begin{tabular}{ccc}
\rmoveio{.4mm}{2}&\rmoveiio{.4mm}{2}&\rmoveiiio{.4mm}{2}\\
$D$\phantom{Mk}$D'$\phantom{Mk}&$D$\phantom{MMM}$D'$&$D$\phantom{MMMk}$D'$\\
I&II&III\\
\multicolumn{3}{c}{Reidemeister moves}
\end{tabular}}
\vspace{0.2cm}
\centerline{
\begin{tabular}{cccc}
\rmovevio{.4mm}{1}&\rmoveviio{.4mm}{1}&\rmoveviiio{.4mm}{1}&\rmovevivo{.4mm}{1}\\
I&II&III&IV\\
\multicolumn{4}{c}{Virtual Reidemeister moves}
\end{tabular}}
\caption{Generalized Reidemeister moves}\label{fgmoves}
\end{figure}
%

\section{The $n$-writhe and the affine index polynomial of a virtual knot}\label{sect:defmain}

In this section, we recall two kinds of invariants of virtual knots, 
the $n$-writhe $J_n(K)$ $(n \neq 0)$ due to Satoh and Taniguchi \cite{rST} and the affine index polynomial $P_K(t)$ due to Kauffman \cite{rkauF}.  These two invariants were introduced independently, and it turns out that they are essentially the same invariant.  

Let $D$ be a virtual knot diagram and $c$ a classical crossing.  
The sign of $c$ is denoted by ${\rm sgn}(c)$. 
Let $D_c$  be a virtual link diagram obtained from $D$ by smoothing at $c$.

The {\it left\/} (or {\it right\/}) {\it counting component\/} for $D$ at $c$ means a component of $D_c$ which is 
denoted by $D_c^{\rm L}$ (or $D_c^{\rm R}$) in Figure~\ref{fig:smooth}.  

\begin{figure}[h]
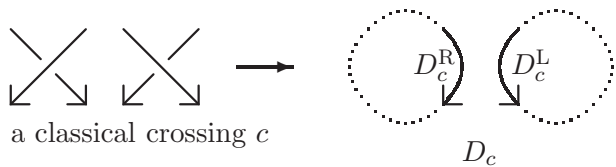

\begin{center}
\smoothLR{.5mm}
\end{center}
\caption{Counting components $D_c^{\rm L}$ and $D_c^{\rm R}$}\label{fig:smooth}
\end{figure}

The {\it over counting component\/} and the {\it under counting component\/} for $D$ at $c$ mean components of $D_c$ which are  
denoted by $D_c^{\rm O}$ and $D_c^{\rm U}$ in Figures~\ref{fig:smoothOUpositive} and  \ref{fig:smoothOUnegative} according to the sign of $c$. In other words,  
$$ 
\begin{array}{rrl}
D_c^{\rm O} = D_c^{\rm R}, & \quad  D_c^{\rm U}= D_c^{\rm L}     & \qquad ({\rm sgn}(c) =+1), \\
D_c^{\rm O} = D_c^{\rm L}, & \quad  D_c^{\rm U}= D_c^{\rm R}     & \qquad ({\rm sgn}(c) =-1). 
\end{array}
$$

\begin{figure}[h]
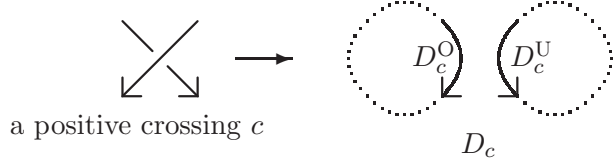

\begin{center}
\smoothOUpositive{.5mm}
\end{center}
\caption{Counting components $D_c^{\rm O}$ and $D_c^{\rm U}$ for positive $c$}\label{fig:smoothOUpositive}
\end{figure}

\begin{figure}[h]
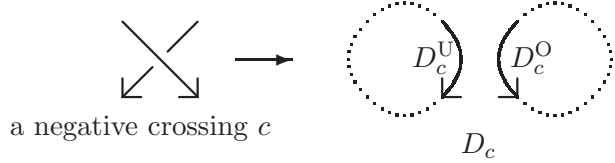

\begin{center}
\smoothOUnegative{.5mm}
\end{center}
\caption{Counting components $D_c^{\rm O}$ and $D_c^{\rm U}$ for negative $c$}\label{fig:smoothOUnegative}
\end{figure}

Let $D_c^{\ast}$ $(\ast \in \{ {\rm L}, {\rm R}, {\rm O}, {\rm U} \})$ be a counting component for $D$ at $c$. 
Let $\gamma_1, \dots, \gamma_k$ be the sequence of classical crossings of $D$ which appear on $D_c^{\ast}$ 
in this order when we go along $D_c^{\ast}$ from a point near $c$ and come back to the point. 
Here if there is a self-crossing of $D_c^{\ast}$, then it appears twice in the sequence $\gamma_1, \dots, \gamma_k$.

We define the {\it flat sign} at $\gamma_i$, denoted by $\iota(\gamma_i, D_c^{\ast})$, as follows. 
$$\iota(\gamma_i, D_c^{\ast})=\left\{
\begin{array}{rl}
{\rm sgn}(\gamma_i)&\text{ if $\gamma_i$ is an under crossing, }\\
-{\rm sgn}(\gamma_i)&\text{ if $\gamma_i$ an over crossing. }
\end{array}\right.
$$

In other words, the {\it flat sign} is $+1$ (or $-1$) if 
 the other arc passing through the classical crossing is oriented from left to right (or right to left) with respect to the orientation of $D_c^{\ast}$.

The {\it left\/}, {\it right\/}, {\it over \/} or {\it under\/} {\it affine index\/} of $D$ at a classical crossing $c$, which we denote by 
${\rm ind}^{\rm L}(c, D)$, ${\rm ind}^{\rm R}(c, D)$,  ${\rm ind}^{\rm O}(c, D)$ or ${\rm ind}^{\rm U}(c, D)$) respectively, is defined by 
$$ {\rm ind}^{\ast}(c, D) = \sum_i \iota(\gamma_i, D_c^{\ast}) \quad \quad 
(\ast \in \{ {\rm L}, {\rm R}, {\rm O}, {\rm U}  \}). $$ 

The {\it index\/} of a classical crossing $c$ in the sense of Satoh and Taniguchi, which we denote by ${\rm Ind}^{\rm ST}(c)$ or ${\rm Ind}^{\rm ST}(c, D)$, is defined by  
$$ 
{\rm Ind}^{\rm ST}(c) =
{\rm ind}^{\rm O}(c, D).
$$

For each  non-zero integer $n$, 
the {\it $n$-writhe\/} $J_n(D)$ of $D$ is defined by 
$$J_n(D)=\sum_{c \in X(D) ~:~ {\rm Ind}^{\rm ST}(c)=n} {\rm sgn}(c),$$ 
where $c$ runs over the set of classical crossings $X(D)$ with ${\rm Ind}^{\rm ST}(c)=n$. 

\begin{thm}[\cite{rST}]
If $D$ and $D'$ represent the same virtual knot, then $J_n(D)= J_n(D')$ for each $n \neq 0$. 
\end{thm}

For a virtual knot $K$ represented by a diagram $D$, the {\it $n$-writhe} $J_n(K)$ of $K$ is defined by $J_n(D)$,  \cite{rST}.  

The affine index polynomial was introduced by Kauffman in \cite{rkauF}.  For a classical crossing $c$ of $D$, 
two integers $W_-(c)$ and  $W_+(c)$ are defined by 
$$ W_+(c) = {\rm ind}^{\rm L}(c, D) \quad \text{and} \quad  W_-(c) = {\rm ind}^{\rm R}(c, D). $$ 
Then $W_D(c)$ is defined by $W_{{\rm sgn}(c)} (c)$, namely,  
$$ 
W_D(c) =
{\rm ind}^{\rm U}(c, D).
$$
The {\it affine index polynomial\/} of $D$ is defined by 
$$ P_D(t) = \sum_{c \in X(D)} {\rm sgn}(c) ( t^{W_D(c)} -1). $$ 

\begin{thm}[\cite{rkauF}]
If $D$ and $D'$ represent the same virtual knot, then $P_D(t)= P_{D'}(t)$. 
\end{thm}

For a virtual knot $K$ represented by a diagram $D$, the {\it affine index polynomial} $P_K(t)$ is defined by $P_D(t)$. 

\begin{lem}[\cite{rkauF,rST}]
For any classical crossing $c$ of a virtual knot diagram $D$,  
$${\rm ind}^{\rm L}(c, D) = - {\rm ind}^{\rm R}(c, D)$$
and
$${\rm ind}^{\rm O}(c, D) = - {\rm ind}^{\rm U}(c, D)$$
\end{lem} 

Using this lemma and definitions, we have the following. 

\begin{lem}\label{indexequivalent}
For any classical crossing $c$ of a virtual knot diagram $D$,  
$$ W_D(c) = - {\rm Ind}^{\rm ST}(c).  $$ 
\end{lem} 

This lemma implies the following. 

\begin{prop}\label{prop:STandP}
Let $n$ be a non-zero integer.  For any virtual knot $K$, 
the coefficient of $t^n$ of the affine index polynomial $P_K(t)$ is $J_{-n}(K)$.  
\end{prop} 

Thus, the  affine index polynomial $P_K(t)$ determines all $J_n(K)$ $(n\neq 0)$. Conversely, 
all $J_n(K)$ $(n\neq 0)$ determine $P_K(t)$, since $P_K(1) = 0$. 
Therefore we may say that for a virtual knot $K$, 
the family of $n$-writhes $\{ J_n(K) \}_{n \neq 0}$ and the affine index polynomial $P_K(t)$ are essentially the same invariant.

\section{Invariants of virtual links using affine indices}\label{sect:links} 

We have discussed invariants for virtual knots. 
In this section, we discuss invariants for virtual links.  

Let $D$ be a virtual link diagram.  
A classical self-crossing of $D$ means a classical crossing of $D$ such that the over path and the under path are on the same component of $D$.  
We denote by ${\rm Self}X(D)$ the set of classical self-crossings of $D$. 

Let $c$ be a classical self-crossing of $D$. 
Let $D_c$ be the virtual link obtained from $D$ by smoothing at $c$, and 
let $D_c^{\rm O}$ and $D_c^{\rm U}$ be the over and under counting components for $D$ at $c$, which are components of $D_c$ as in Figure~\ref{fig:smoothOUpositive} or \ref{fig:smoothOUnegative}. 

The {\it over affine index} of $D$ at $c$, denoted by ${\rm ind}^{\rm O}(c)$ or  ${\rm ind}^{\rm O}(c, D)$, and the {\it under affine index} of $D$ at $c$, denoted by ${\rm ind}^{\rm U}(c)$ or ${\rm ind}^{\rm U}(c, D)$, are also defined as the the same way with those in the case where $D$ is a virtual knot diagram.

\begin{definition}\label{def:virtuallinkinvariants}{\rm 
Let $D$ be a virtual link diagram. Let $n$ be a non-zero integer. 
\begin{itemize}
\item[(1)] 
The {\it over $n$-writhe} of $D$ is defined by 
$$ 
J^{\rm O}_n(D) = \sum_{c \in {\rm Self}X(D) ~:~ 
{\rm ind}^{\rm O}(c) = n, ~ 
{\rm ind}^{\rm U}(c) \neq 0} 
{\rm sgn}(c). 
$$ 
\item[(2)] 
The {\it under $n$-writhe} of $D$ is defined by 
$$ 
J^{\rm U}_n(D) = \sum_{c \in {\rm Self}X(D) ~:~ 
{\rm ind}^{\rm U}(c) = n, ~ 
{\rm ind}^{\rm O}(c) \neq 0} 
{\rm sgn}(c). 
$$ 
\item[(3)] 
The {\it over affine index polynomial} of $D$ is defined by 
$$ 
P^{\rm O}_D(t) = \sum_{c \in {\rm Self}X(D) ~:~ 
{\rm ind}^{\rm U}(c) \neq 0} 
{\rm sgn}(c) (t^{{\rm ind}^{\rm O}(c)} -1). 
$$ 
\item[(4)] 
The {\it under affine index polynomial} of $D$ is defined by 
$$ 
P^{\rm U}_D(t) = \sum_{c \in {\rm Self}X(D) ~:~ 
{\rm Ind}^{\rm O}(c) \neq 0} 
{\rm sgn}(c) (t^{{\rm ind}^{\rm U}(c)} -1). 
$$ 
\item[(5)] 
The {\it over-under affine index polynomial} of $D$ is defined by 
$$ 
P^{\rm OU}_D(s, t) = \sum_{c \in {\rm Self}X(D)} 
{\rm sgn}(c) (s^{{\rm ind}^{\rm O}(c)} -1) (t^{{\rm ind}^{\rm U}(c)} -1). 
$$ 
\end{itemize}
}
\end{definition}

\begin{thm}\label{virtuallinkinvariants}
Let $D$ be a virtual link diagram. 
Let $n$ be a non-zero integer. Then 
$J^{\rm O}_n(D)$, $J^{\rm U}_n(D)$, $P^{\rm O}_D(t)$, $P^{\rm U}_D(t)$ and  
$P^{\rm OU}_D(s, t)$ are virtual link invariants. 
\end{thm}

By definition, $J^{\rm O}_n(D)$ and $J^{\rm U}_n(D)$ are the coefficients of $t^n$ in 
$P^{\rm O}_D(t)$ and  $P^{\rm U}_D(t)$, respectively.  Therefore the invariance of $J^{\rm O}_n(D)$ and $J^{\rm U}_n(D)$  comes from that of $P^{\rm O}_D(t)$ and  $P^{\rm U}_D(t)$.  

\begin{proof} 
Virtual Reidemeister moves do not change $P^{\rm O}_D(t)$, $P^{\rm U}_D(t)$ and $P^{\rm OU}_D(s, t)$. Thus, it is sufficient to show that $P^{\rm O}_D(t)$, $P^{\rm U}_D(t)$ and $P^{\rm OU}_D(s, t)$ are preserved under Reidemeister moves. 

(i) Let $D'$ be obtained from $D$ by a Reidemesiter move I as in Figure~\ref{fgmoves}, where we assume 
$X(D') = X(D) \cup \{ c_1 \}$. 
For $c \in X(D) = X(D') \setminus \{ c_1 \}$, we see that 
${\rm ind}^{\rm O}(c)$ and ${\rm ind}^{\rm U}(c)$ do not change.  
Since 
${\rm ind}^{\rm O}(c_1)=0$ or ${\rm ind}^{\rm U}(c_1)=0$,  
we have $P^{\rm O}_D(t) = P^{\rm O}_{D'}(t)$, $P^{\rm U}_D(t) = P^{\rm U}_{D'}(t)$ and $P^{\rm OU}_D(s, t) = P^{\rm OU}_{D'}(s, t)$.

(ii) Let $D'$ be obtained from $D$ by a Reidemesiter move II as in Figure~\ref{fgmoves}. 
For $c \in X(D) = X(D') \setminus \{ c_1, c_2 \}$, we see that  
${\rm ind}^{\rm O}(c)$ and ${\rm ind}^{\rm U}(c)$ do not change. 
Note that ${\rm ind}^{\rm O}(c_1) = {\rm ind}^{\rm O}(c_2)$ and ${\rm ind}^{\rm U}(c_1) = {\rm ind}^{\rm U}(c_2)$. 
Since ${\rm sgn}(c_1) = - {\rm sgn}(c_2)$, we see that the contribution of 
$c_1$ cancels with that of $c_2$.  

(iii) Let $D'$ be obtained from $D$ by a Reidemesiter move III as in Figure~\ref{fgmoves}. 
For $c \in X(D) \setminus \{ c_1, c_2, c_3 \} = X(D') \setminus \{ c'_1, c'_2, c'_3 \}$, we see that  
${\rm ind}^{\rm O}(c)$ and ${\rm ind}^{\rm U}(c)$ do not change. 
Note that, for $i \in \{1, 2, 3\}$, ${\rm sgn}(c_i) = {\rm sgn}(c'_i)$ and 
$$ 
{\rm ind}^{\rm O}(c_i, D) = {\rm ind}^{\rm O}(c'_i, D') \quad \text{and} \quad   
{\rm ind}^{\rm U}(c_i, D) = {\rm ind}^{\rm U}(c'_i, D').  $$
Thus we have $P^{\rm O}_D(t) = P^{\rm O}_{D'}(t)$, $P^{\rm U}_D(t) = P^{\rm U}_{D'}(t)$ and $P^{\rm OU}_D(s, t) = P^{\rm OU}_{D'}(s, t)$. 
\end{proof}

For a virtual link $L$ represented by a diagram $D$, we define $J^{\rm O}_n(L)$, $J^{\rm U}_n(L)$, $P^{\rm O}_L(t)$, $P^{\rm U}_L(t)$ and $P^{\rm OU}_L(s, t)$ by 
$J^{\rm O}_n(D)$, $J^{\rm U}_n(D)$, $P^{\rm O}_D(t)$, $P^{\rm U}_D(t)$ and $P^{\rm OU}_D(s, t)$, respectively.  

\begin{rem}{\rm 
(1) 
When $K$ is a virtual knot, we have $J_n(K) = J^{\rm O}_n(K) = - J^{\rm U}_n(K)$ for $n \neq 0$, and 
$ P_K(t) = P^{\rm U}_K(t) = P^{\rm O}_K(t^{-1})$.  

(2) Let $\ast \in \{ {\rm O}, {\rm U} \}$. For any virtual link $L$, the polynomial $P^{\ast}_L(t)$ determines all $J^{\ast}_n(L)$ $(n\neq 0)$. Conversely all $J^{\ast}_n(L)$ $(n\neq 0)$ determine $P^{\ast}_L(t)$. 

(3) For any virtual link $L$, 
$$ 
P^{\rm O}_L(t) = - P^{\rm OU}_L(t, 0) \quad \text{\rm and} \quad 
P^{\rm U}_L(t) = - P^{\rm OU}_L(0, t).
$$
}\end{rem}

\begin{rem}\label{rem:compatible}{\rm 
A virtual link diagram $D$ is called {\it compatible} if  
for any classical crossing $c$,  
${\rm ind}^{\rm L}(c, D) = - {\rm ind}^{\rm R}(c, D)$ or equivalently if ${\rm ind}^{\rm O}(c, D) = - {\rm ind}^{\rm U}(c, D)$. 
A virtual knot diagram is always compatible.  
In~\cite{rKauffman2018} Kauffman stated that the affine index polynomial $P_K(t)$ of a virtual knot is naturally generalized to 
the affine index polynomial $P_L(t)$ of a compatible virtual link. 
The invariant $P_L(t)$ is studied in \cite{rPetit2020}.  
For a compatible virtual link $L$, we have $$P_L(t) = P^{\rm U}_L(t) = P^{\rm O}_L(t^{-1}), $$  
namely, our invariants $P^{\rm O}_L(t)$ and $P^{\rm U}_L(t)$ are essentially the same with $P_L(t)$.  
}\end{rem}

\begin{rem}{\rm 
An invariant of a $2$-component virtual link, called the Wriggle polynomial,  is defined in \cite{rFolwacznyKauffman}. Using it, the affine index polynomial $P_K(t)$ of a virtual knot is understood in terms of the linking number, \cite{rFolwacznyKauffman}.  Cheng and Gao \cite{rChengGao2013} investigated an invariant of a $2$-component virtual link, which is related to the linking number. These invariants for $2$-component virtual links are computed by using crossings between the two components, while our polynomial invariants $P_L^{\rm O}(t)$, $P_L^{\rm U}(t)$ and $P_L^{\rm OU}(s,t)$ are defined and computed by using self-crossings.  
}\end{rem}

The following two remarks were suggested by the referee.

\begin{rem}\label{rem:Xu}{\rm 
An invariant of a virtual link, denoted by $Wr_{L}(t)$, is defined in Definition~5.6 of  \cite{rXu}.  In order to define $Wr_{L}(t)$,  an index denoted by ${\rm Ind}'(c)$ of a self-crossing $c$ of a diagram is introduced in \cite{rXu}. For a virtual link diagram $D$, 
$$Wr_{D}(t) = \sum_{c \in {\rm Self}X(D) ~:~ 
{\rm Ind}(c) \neq 0} 
{\rm sgn}(c) (t^{{\rm Ind}'(c)} -1). 
$$ 
This invariant and our invariant 
$$ 
P^{\rm O}_D(t) = \sum_{c \in {\rm Self}X(D) ~:~ 
{\rm ind}^{\rm U}(c) \neq 0} 
{\rm sgn}(c) (t^{{\rm ind}^{\rm O}(c)} -1). 
$$ 
are essentially the same.  In fact, since it holds that 
\begin{itemize}
\item ${\rm Ind}(c) = 0$ if and only if 
${\rm ind}^{\rm U}(c) = 0$ or 
${\rm ind}^{\rm O}(c) = 0$, and 
\item 
${\rm Ind}'(c) = {\rm ind}^{\rm O}(c)$, 
\end{itemize}
we obtain 
$$ P^{\rm O}_D(t)  = Wr_{D}(t) -  Wr_{D}(1).$$ 
In other words, $P^{\rm O}_L(t)$ is the normalization of $Wr_{L}(t)$ which vanishes at $t=1$. 
}\end{rem}

\begin{rem}\label{rem:linking}{\rm 
When we discuss ordered virtual links, we may consider component-wise linking numbers.  
Let ${\rm lk}(K_i, K_j)$ be the {\rm $(i,j)$-linking number} of an ordered $n$-component virtual link $L = K_1 \cup \dots \cup K_n$, which is defined to be the sum of signs of nonself-crossings where the over and under paths come from $K_i$ and $K_j$, respectively. Moreover, let 
$${\rm vlk}(K_i, K_j) = {\rm lk}(K_j, K_i) - {\rm lk}(K_i, K_j) $$ 
be the {\it virtual $(i,j)$-linking number} of a pair $(K_i, K_j)$.  Put 
$$\lambda_i = \lambda_i(L) = 
\sum_{j=1}^n  {\rm vlk}(K_i, K_j) $$ 
for each $i$. Let $S(L)_i$ be the set of self-crossings whose over and under paths come from $K_i$ both. Then for any $c \in S(L)_i$ satisfies 
$${\rm ind}^{\rm O}(c) + {\rm ind}^{\rm U}(c)  = \lambda_i.$$  
In \cite{rXu}, Xu splits the writhe polynomial $Wr_{L}(t)$ into $\sum_{i=1}^n W_i(t)$, where 
$$W_i(t)  = \sum_{c \in S(L)_i ~:~ 
{\rm Ind}(c) \neq 0} 
{\rm sgn}(c) (t^{{\rm Ind}'(c)} -1). 
$$ 
Therefore, for an ordered virtual link, the ordered set $\{ W_i(t) \}_{i=1, \dots, n}$ is an invariant which is a refinement of $Wr_L(t)$. 

The same applies to 
$P^{\rm O}_L(t)$ and $P^{\rm U}_L(t)$.  We remark that ${\rm Ind}(c) \neq 0$ if and only if 
${\rm ind}^{\rm \ast}(c) \neq 0, \lambda_i$ for each $\ast \in \{ {\rm O}, {\rm U} \}$. 
For $\ast \in \{ {\rm O}, {\rm U} \}$, put 
$$ 
P^{\rm \ast}_L(t)_i = \sum_{c \in S(L)_i ~:~ 
{\rm ind}^{\rm \ast}(c) \neq 0, \lambda_i} 
{\rm sgn}(c) (t^{{\rm ind}^{\rm \ast}(c)} -1). 
$$ 
Since $P^{\rm \ast}_L(t) = \sum_{i=1, \dots, n} P^{\rm \ast}_L(t)_i$ holds, 
$\{ P^{\rm \ast}_L(t)_i \}_{i=1, \dots, n}$ is an invariant of an ordered virtual link, which is 
a refinement of $P^{\rm \ast}_L(t)$. 
Furthermore, it holds that 
\begin{eqnarray*}
P^{\rm U}_L(t)_i 
&= & \sum_{c \in S(L)_i ~:~ 
{\rm ind}^{\rm U}(c) \neq 0, \lambda_i} 
{\rm sgn}(c) (t^{{\rm ind}^{\rm U}(c)} -1) \\ 
&= & \sum_{c \in S(L)_i ~:~ 
{\rm ind}^{\rm O}(c) \neq 0, \lambda_i} 
{\rm sgn}(c) (t^{\lambda_i - {\rm ind}^{\rm O}(c)} -1) \\
&= & \sum_{c \in S(L)_i ~:~ 
{\rm ind}^{\rm O}(c) \neq 0, \lambda_i} 
{\rm sgn}(c) (t^{\lambda_i - {\rm ind}^{\rm O}(c)} - t^{\lambda_i}) \\ 
& & 
- \sum_{c \in S(L)_i ~:~ 
{\rm ind}^{\rm O}(c) \neq 0, \lambda_i} 
{\rm sgn}(c) (t^{\lambda_i} -1 ) \\ 
&= & 
t^{\lambda_i} \cdot P^{\rm O}_L(t^{-1})_i - \left(  
\sum_{k \neq 0, \lambda_i} J^{\rm O}_k(L) \right) 
\cdot ( t^{\lambda_i} -1 ).  
\end{eqnarray*}
In this seise, if we assume the information $\{ \lambda_i \}_{i=1, \dots, n}$ regarding virtual linking numbers is known, we can say that invariants $\{ P^{\rm O}_L(t)_i \}_{i=1, \dots, n}$ and $\{ P^{\rm U}_L(t)_i \}_{i=1, \dots, n}$ are essentially the same. 

Similarly, when we split the polynomial 
 $P^{\rm OU}_L(s,t)$ into $\sum_{i=1}^n P^{\rm OU}_L(s,t)_i$, it holds that 
 \begin{eqnarray*}
P^{\rm OU}_L(s,t)_i 
&= & \sum_{c \in S(L)_i} 
{\rm sgn}(c) (s^{{\rm ind}^{\rm O}(c)} -1)  (t^{{\rm ind}^{\rm U}(c)} -1) \\  
&= & \sum_{c \in S(L)_i ~:~ 
{\rm ind}^{\rm O}(c) \neq 0, \lambda_i} 
{\rm sgn}(c) ( s^{{\rm ind}^{\rm O}(c)} t^{{\rm ind}^{\rm U}(c)} - 
s^{{\rm ind}^{\rm O}(c)} - t^{{\rm ind}^{\rm U}(c)} +1)  \\ 
&= & \sum_{c \in S(L)_i ~:~ 
{\rm ind}^{\rm O}(c) \neq 0, \lambda_i} 
{\rm sgn}(c) ( s^{{\rm ind}^{\rm O}(c)} t^{ \lambda_i - {\rm ind}^{\rm O}(c)} - 1) 
- P^{\rm O}_L(t)_i - P^{\rm U}_L(t)_i \\ 
&= & 
t^{\lambda_i} \cdot P^{\rm O}_L(st^{-1})_i   
- \left(  
\sum_{k \neq 0, \lambda_i} J^{\rm O}_k(L) \right) 
\cdot  ( t^{\lambda_i} -1 ) 
- P^{\rm O}_L(t)_i - P^{\rm U}_L(t)_i
\end{eqnarray*}
In this seise, if we assume  the information $\{ \lambda_i \}_{i=1, \dots, n}$  
regarding virtual linking numbers is known,  
 $\{ P^{\rm OU}_L(s,t)_i \}$ is obtained from $\{ P^{\rm O}_L(t)_i \}$.  
}\end{rem}

\section{Alexander numberings and cut systems}\label{sect:cut} 

In this section we first recall Alexander numberings (\cite{rboden}), oriented cut points and cut systems (\cite{rkn2019}).  Then 
we discuss the invariants introduced in the previous section in terms of cut systems.

Let $D$ be a virtual link diagram. 
{\it A semi-arc} of $D$ means either an immersed arc in a component of $D$ whose endpoints are classical crossings or an immersed  loop missing classical crossings of $D$.  There may exist virtual crossings of $D$ on a semi-arc. 

Let $m$ be a non-negative integer. 
An {\it Alexander numbering} (or a {\it mod $m$ Alexander numbering}) 
of $D$ is an assignment of an element of $\mathbb{Z}$ (or $\mathbb{Z}_m = \mathbb{Z} / m \mathbb{Z}$) 
to each semi-arc of $D$ such that for each classical crossing 
the numbers assigned to the semi-arcs around it are as  shown in Figure~\ref{fgalexnum} for some $i \in \mathbb{Z}$ (or $i \in \mathbb{Z}_m$). 
The numbers around a virtual crossing are as in Figure~\ref{fgalexnumv}, since a virtual crossing is just a self-crossing of a semi-arc or a crossing of two semi-arcs. 

\vspace{0.3cm}
\begin{figure}[h]
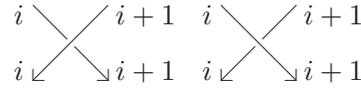

\centerline{
\alexnum{.5mm}
}
\caption{Alexander numbering}\label{fgalexnum}
\end{figure}

\vspace{0.3cm}
\begin{figure}[h]
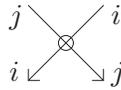

\centerline{
\alexnumv{.5mm}
}
\caption{Alexander numbering around a virtual crossing}\label{fgalexnumv}
\end{figure}

Not every virtual link diagram admits an Alexander numbering, while every classical link diagram does.    The virtual knot diagram depicted in Figure~\ref{fgexAlexN2} (i) does not admit an Alexander numbering, and the virtual knot diagram 
in Figure~\ref{fgexAlexN2} (ii) does. 

\begin{figure}[h]
\centerline{
\includegraphics[width=7.cm]{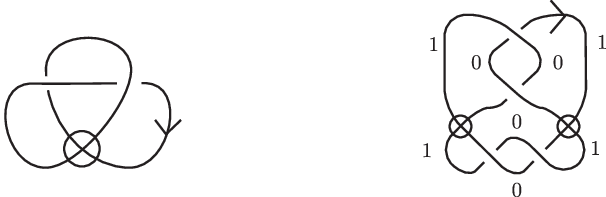}
}
\centerline{(i)\hspace{4.5cm}(ii)}
\caption{(i) A diagram which does not admit an Alexander numbering and (ii) a diagram which does}\label{fgexAlexN2}
\end{figure}

Figure~\ref{fgexMAlexN1} shows a mod $3$ Alexander numbering, which is not an Alexander numbering.

\begin{figure}[h]
\centerline{
\includegraphics[width=4.5cm]{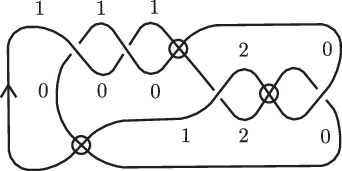}
}
\caption{A mod $3$ Alexander numbering of a virtual knot diagram}\label{fgexMAlexN1}
\end{figure}

A virtual link diagram  is {\it almost classical} (or {\it mod $m$ almost classical\/}) 
if it admits an Alexander numbering (or a mod $m$ Alexander numbering). 
A virtual link $L$ is {\it almost classical\/} (or {\it mod $m$ almost classical})  
if there is an almost classical (or mod $m$ almost classical) virtual link diagram representing $L$.

Boden, Gaudreau, Harper, Nicas and White~\cite{rboden} studied mod $m$ almost classical virtual knots. 
It is shown in \cite{rboden} that for a mod $m$ almost classical virtual knot $K$, if $D$ is a minimal (in numbers of classical crossings) virtual knot diagram of $K$, then $D$ is mod $m$ almost classical.

An {\it oriented cut point} is a point on a semi-arc of $D$ at which a local orientation of the semi-arc is given (\cite{rkn2019}). In this paper we denote it by a small triangle on the semi-arc as in Figure~\ref{fgorientedcutpt}.  
     An oriented cut point is called  {\it coherent} (or.  {\it incoherent}) if the local orientation indicated by the cut point  is the same with  (or the opposite to) the orientation of the virtual link diagram $D$.

\vspace{0.3cm} 
\begin{figure}[h]
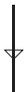

\centerline{
\orientedcutpt{.6mm}
}
\caption{An oriented cut point}\label{fgorientedcutpt}
\end{figure}

Let $D$ be  a virtual link diagram.   A finite set of oriented cut points of $D$ is called 
an {\it oriented cut system} or simply a {\it cut system} if 
$D$ admits an Alexander numbering such that 
at each oriented cut point,  the number increases by one in the direction of the 
oriented cut point (Figure~\ref{fgalexnumorict}).   
Such an Alexander numbering is called an {\it Alexander numbering of a virtual link diagram with a cut system}. 
See Figure~\ref{fgExcutPtAlex} for examples. 

\vspace{0.3cm} 
\begin{figure}[h]
\centerline{
\alexnum{.5mm}\hspace{.6cm}\alexnumv{.5mm}\hspace{.6cm}\alexnumcp{.5mm}
}
\caption{Alexander numbering of a virtual link diagram with a cut system}\label{fgalexnumorict}
\end{figure}
\vspace{0.3cm} 
\begin{figure}[h]
\centerline{
\includegraphics[width=8.5cm]{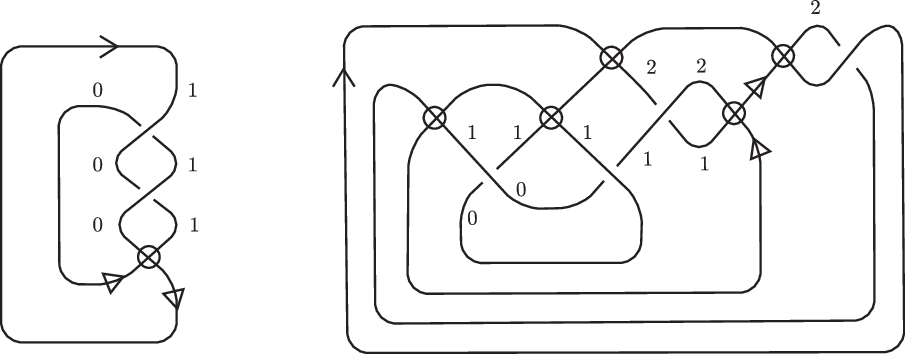}
}
\caption{Alexander numberings of virtual knot diagrams with cut systems}\label{fgExcutPtAlex}
\end{figure}

Every virtual link diagram admits a cut system. For example, we can construct a cut system for a given virtual link diagram as below: 

\begin{example}{\rm 
(1) The {\it binary cut system} is a cut system which is obtained by giving  
 two oriented cut points as in Figure~\ref{fg:canocutsys} around each classical crossing.  

\vspace{0.3cm}
\begin{figure}[h]
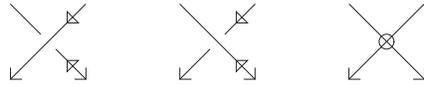

\centerline{
\canocutsysic{.5mm}{1}}
\caption{The binary cut system}\label{fg:canocutsys}
\end{figure}

The binary cut system is a cut system, since we have an 
Alexander numbering using only $0$ and $1$ as in Figure~\ref{fg:canocutsysn} for each crossing.   

\vspace{0.3cm}

\begin{figure}[h]
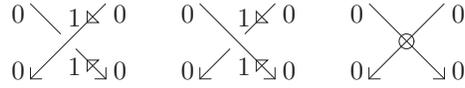

\centerline{
\canocutsysic{.5mm}{2}}
\caption{The binary cut system with an Alexander numbering}\label{fg:canocutsysn}
\end{figure}

(2) 
The  {\it flat cut system} of $D$ is a cut system which is obtained by giving  
 two oriented cut points as in Figure~\ref{fg:stncutsys} around each virtual crossing.

\begin{figure}[h]
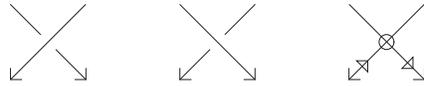

\centerline{
\stnocutsysi{.5mm}{1}}
\caption{The flat cut system}\label{fg:stncutsys}
\end{figure}

The flat cut system is a cut system:  Consider a collection of embedded oriented loops in $\mathbb{R}^2$ obtained from $D$ by smoothing at all (classical/virtual) crossings. We can assign integers to the loops such that an Alexander numbering of $D$ is induced from the integers.  See Figure~\ref{fg:stncutsysn} for a local picture around each crossing. 

\vspace{0.3cm}
\begin{figure}[h]
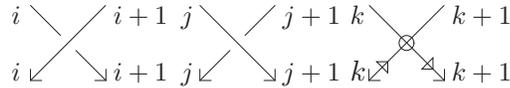

\centerline{
\stnocutsysi{.5mm}{2}}
\caption{The flat cut system with an Alexander numbering}\label{fg:stncutsysn}
\end{figure}

}\end{example}

The local transformations of oriented cut points depicted in Figure~\ref{fgOrientedCutmv} are called 
 {\it oriented cut point moves}.  For a virtual link diagram with a cut system, the result by an oriented cut point move  is also a cut system of the same virtual link diagram.   
Note that the move III$'$ depicted in Figure~\ref{fgOrientedCutmv} is obtained from the move III modulo the moves II. 

\vspace{0.3cm}
\begin{figure}[h]
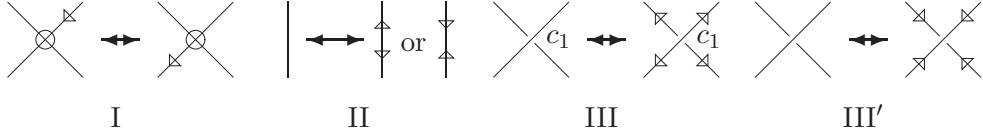

\centerline{
\begin{tabular}{cccc}
\ocpmovei{.5mm}&\ocpmoveii{.5mm}&\ocpmoveiiia{.5mm}{2}&\ocpmoveiiib{.5mm}{1}\\
I&II&III&III$'$\\
\end{tabular}
}
\caption{Oriented cut point moves}\label{fgOrientedCutmv}
\end{figure}

\begin{thm}[\cite{rkn2019}] \label{thm:cutpointmove}
Two cut systems of the same virtual link diagram are related by a finite sequence of oriented cut point moves. 
\end{thm}


\begin{cor}[\cite{rkn2019}] \label{cor:numbers}
Let $D$ be a virtual link diagram and  $C$  a cut system of $D$. The number of coherent cut points of $C$ equals that of incoherent cut points of $C$. 
\end{cor}

\begin{proof}
The binary cut system for $D$ has the property that the number of coherent cut points  equals that of incoherent cut points.  Since each oriented cut point move preserves this property, by Theorem~\ref{thm:cutpointmove} we see that any cut system has the property.  
\end{proof}

Let $D$ be a virtual link diagram and $C$ a cut system of $D$. 
Let $c$ be a classical self-crossing of $D$ and let $D_c$ be the virtual link diagram obtained from $D$ by smoothing at $c$. 
Let $D_c^{\rm O}$ and $D_c^{\rm U}$ be the over and under counting components for $D$ at $c$.   

For 
$\ast \in \{ {\rm O}, {\rm U} \}$, 
let $\rho^{\ast}(c,C)$ be the 
number of incoherent cut points of $C$ appearing on $D_c^{\ast}$ minus that of coherent ones.

\begin{lem}\label{proofmain2}
Let $C$ and $C'$ be cut systems of $D$. Let $\ast \in \{ {\rm O}, {\rm U} \}$. 
Then $\rho^{\ast}(c, C) = \rho^{\ast}(c, C')$ for $c \in {\rm Self}X(D)$.
\end{lem}

\begin{proof}
We show that $\rho^{\ast}(c, C)$ is preserved under oriented cut point moves (Figure~\ref{fgOrientedCutmv}). It is obvious that $\rho^{\ast}(c, C)$ is preserved under moves I and II.  
Consider a move III and  let $c_1$ be the classical crossing as in the figure.  If $c \neq c_1$, then it is obvious that $\rho^{\ast}(c, C)$ is preserved. If $c = c_1$, then 
 the  number of coherent cut points on $D_c^{\rm O}$ (or $D_c^{\rm U}$) increases by one and that of incoherent cut points increases by one. Thus $\rho^{\ast}(c, C)$ is also preserved in this case. 
\end{proof}

By this lemma, for a classical self-crossing $c$ of a virtual link diagram $D$, 
we may define $\rho^{\ast}(c)$ by $\rho^{\ast}(c,C)$, where $\ast \in \{ {\rm O}, {\rm U} \}$ and $C$ is any cut system of $D$.

\begin{thm}\label{proofmain1}
For a classical self-crossing $c$ of a virtual link diagram $D$, 
$\rho^{\ast}(c) = {\rm ind}^{\ast}(c)$, where $\ast \in \{ {\rm O}, {\rm U} \}$.  
\end{thm}

\begin{proof}
Let $C$ be the binary cut system of $D$. See Figure~\ref{fg:canocutsys}. 
Near the crossing $c$, a pair of coherent cut point and an incoherent cut point appear. They do not contribute to $\rho^{\ast}(c)$. 

Let $\gamma_1, \dots, \gamma_k$ be the sequence of classical crossings of $D$ 
appearing on $D_c^{\ast}$ in this order when we go along $D_c^{\ast}$ from a point near $c$ and come back to the point. Let $\gamma$ be one of $\gamma_1, \dots, \gamma_k$.  

Suppose that we go through $\gamma$ as an over crossing. 
If ${\rm sgn}(\gamma)$ is $+1$ (or  $-1$), then the oriented cut point on $D_c^{\ast}$ near $\gamma$ is coherent (or incoherent). 

Suppose that we go through $\gamma$ as an under crossing.  
If ${\rm sgn}(\gamma)$ is $+1$ (or  $-1$), then the oriented cut point on $D_c^{\ast}$ near $\gamma$ is incoherent (or coherent).

Thus we have $\rho^{\ast}(c) = {\rm ind}^{\ast}(c)$. 
\end{proof}

\begin{definition}\label{def:cutlink}{\rm 
Let $D$ be a virtual link diagram. Let $n$ be a non-zero integer. 
\begin{itemize}
\item[(1)] 
The {\it over $n$-cut-writhe} of $D$ is defined by 
$$ 
J'^{\rm O}_n(D) = \sum_{c \in {\rm Self}X(D) ~:~ 
\rho^{\rm O}(c) = n, ~ 
\rho^{\rm U}(c) \neq 0} 
{\rm sgn}(c). 
$$ 
\item[(2)] 
The {\it under $n$-cut-writhe} of $D$ is defined by 
$$ 
J'^{\rm U}_n(D) = \sum_{c \in {\rm Self}X(D) ~:~ 
\rho^{\rm U}(c) = n, ~ 
\rho^{\rm O}(c) \neq 0} 
{\rm sgn}(c). 
$$ 
\item[(3)] 
The {\it over affine cut-index polynomial} of $D$ is defined by 
$$ 
P'^{\rm O}_D(t) = \sum_{c \in {\rm Self}X(D) ~:~ 
\rho^{\rm U}(c) \neq 0} 
{\rm sgn}(c) (t^{\rho^{\rm O}(c)} -1). 
$$ 
\item[(4)] 
The {\it under affine cut-index polynomial} of $D$ is defined by 
$$ 
P'^{\rm U}_D(t) = \sum_{c \in {\rm Self}X(D) ~:~ 
\rho^{\rm O}(c) \neq 0} 
{\rm sgn}(c) (t^{\rho^{\rm U}(c)} -1). 
$$ 
\item[(5)] 
The {\it over-under cut-affine index polynomial} of $D$ is defined by 
$$ 
P'^{\rm OU}_D(s, t) = \sum_{c \in {\rm Self}X(D)} 
{\rm sgn}(c) (s^{\rho^{\rm O}(c)} -1) (t^{\rho^{\rm U}(c)} -1). 
$$ 
\end{itemize}
}
\end{definition}

Then by Theorem~\ref{proofmain1} we have the following.

\begin{cor}\label{thm:cutlink} 
Let $D$ be a virtual link diagram. Then 
$J^{\ast}_n(D) = J'^{\ast}_n(D)$ $(n \neq 0)$ and 
$P^{\ast}_D(t) = P'^{\ast}_D(t)$ for $\ast \in \{ {\rm O}, {\rm U} \}$, and 
$P'^{\rm OU}_D(s, t) = P'^{\rm OU}_D(s, t)$. 
\end{cor}

Definition~\ref{def:cutlink} does not introduce new invariants, but it gives an alternative definition or a method of computing the invariants defined in Definition~\ref{def:virtuallinkinvariants} by using cut systems.  
When a virtual link diagram is given with a cut system, we can compute the invariants somehow easier than using the original definition. 

Now we consider the case of a virtual knot. 

For a classical crossing $c$ of a virtual knot diagram $D$, we define $\rho(c)$ or $\rho(c, D)$  by $\rho^{\rm O}(c)$. 
Note that $\rho^{\rm U}(c) = - \rho^{\rm O}(c)$, since  ${\rm ind}^{\rm U}(c) = - {\rm ind}^{\rm O}(c)$.  

\begin{definition}{\rm 
For each non-zero integer $n$, 
we define the {\it $n$-cut-writhe} of $D$, denoted by $J'_n(D)$, by 
$$J'_n(D)= \sum_{c \in X(D) ~:~ \rho(c)=n} {\rm sgn}(c).$$ 
}
\end{definition}

By definition, for a virtual knot diagram $D$, $ J'_n(D) = J'^{\rm O}_n(D)$.  
By Theorem~\ref{proofmain1}, we have the following.

\begin{cor}\label{thm:cutknot} 
Let $n$ be a non-zero integer. 
For any virtual knot diagram $D$, $J'_n(D) = J_n(D)$. 
In particular, the $n$-cut-writhe is a virtual knot invariant.  
\end{cor}

Again, we note that the $n$-cut-writhe $J'_n(D)$ is nothing more than the $n$-writhe $J_n(D)$. However, when a virtual knot diagram is given with a cut system, it gives an alternative method of computing the $n$-writhe. 

\begin{example}{\rm 
Let $K$ be a virtual knot represented by a diagram depicted in Figure~\ref{fgexpolyinv} (left), which is mod  $3$ almost classical. 
It has a cut system as in the figure (right). Let $c_1, c_2, c_3$ and $c_4$ be the classical crossings of the diagram from the top. Then $\rho(c_1)= 3$, $\rho(c_2) = -3$, $\rho(c_3) =0$ and $\rho(c_4)=0$. Thus, we have $J_3(K) = J'_3(D)=1$, $J_{-3}(K)= J'_{-3}(D)=1$ and $J_n(K)= J'_n(D)=0$ for any non-zero integer $n$ with $n \neq \pm 3$. 

\begin{figure}[h]
\centerline{
\includegraphics[width=5cm]{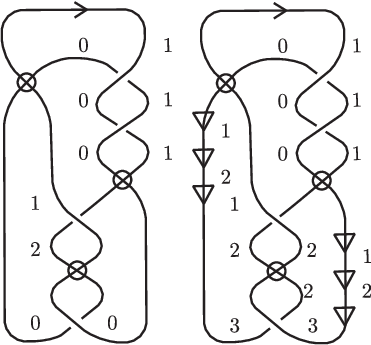}}
\caption{A mod $3$ almost classical virtual knot diagram (left) and a cut system of it (right)}\label{fgexpolyinv}
\end{figure}
}
\end{example}

\begin{prop}
Let $K$ be an almost classical virtual knot. Then we have $J_n(K)=0$ for any non-zero integer $n$.
\end{prop}

\begin{proof}
Let $D$ be an  almost classical virtual knot diagram representing $K$.  
Since $D$ has an empty cut system, $\rho(c)=0$ for $c \in X(D)$ and hence $J_n(K) = J_n(D)= J'_n(D)=0$. 
\end{proof}

For example, the virtual knot diagram depicted in Figure~\ref{fgexAlexN2} (ii) is almost classical. Thus 
$\rho(c)=0$ for $c \in X(D)$ and hence the $n$-writhe is zero for $n \neq 0$.

\begin{prop}
Let $K$ be a mod $m$ almost classical virtual knot. Then we have $J_n(K)=0$ 
unless $ n \equiv 0 \pmod{m}$.
\end{prop}

\begin{proof}
Let $D$ be a mod $m$ almost classical virtual knot diagram representing $K$. Then 
$\rho(c) \equiv 0 \pmod{m}$ for any $c \in X(D)$, which implies that $J'_n(D) = 0 $ unless $ n \equiv 0 \pmod{m}$. 
Thus $J_n(K) = J_n(D) = J'_n(D) = 0$  unless $ n \equiv 0 \pmod{m}$.  
\end{proof}

For example, let $K$ be  a virtual knot represented by the diagram $D$ depicted in Figure~\ref{fgexpolyinv}. Since $J_3(D) = 1$, we see that the virtual knot $K$ is not mod $m$ almost classical unless $m =3$ or $m=1$.

\section{Twisted knots and links}\label{sect:twisted}

A {\it twisted knot} (or {\it link}) {\it diagram} is a virtual knot (or link) diagram possibly with some bars on its arcs. 
In Figure~\ref{fgtwisted}, three twisted knot diagrams are depicted. 
{\it Twisted moves} are local moves depicted in Figure~\ref{fgtmoves}.

\begin{figure}[h]
\centerline{
\begin{tabular}{ccc}
\includegraphics[width=2.5cm, angle=90]{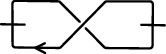}&\includegraphics[width=2.5cm,angle=0]{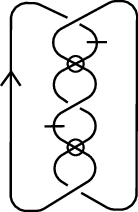}&
\includegraphics[width=4cm, angle=0]{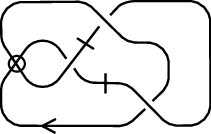}\\
$D_1$&$D_2$&$D_3$\\
(i)&(ii)&(iii)\\
\end{tabular}
}
\caption{Twisted knot diagrams}\label{fgtwisted}
\end{figure}

\begin{figure}[h]
\centerline{
\begin{tabular}{ccc}
\rmoveti{.5mm}&\rmovetiio{.5mm}&\rmovetiiio{.5mm}{2}\\
I& II &III \\
\end{tabular}}
\caption{Twisted moves}\label{fgtmoves}
\end{figure}

Two twisted  knot (or link) diagrams $D$ and $D'$ are said to be {\it equivalent} if they are related by a finite sequence of generalized Reidemeister moves (Figure~\ref{fgmoves}), twisted moves (Figure~\ref{fgtmoves}) and isotopies of $\mathbb{R}^2$.  
A {\it twisted knot\/} (or {\it link})   is an equivalence class of twisted knot (or link) diagrams. 

By definition, if two classical link diagrams are equivalent as classical links then they are equivalent as virtual links.  
If two virtual link diagrams are equivalent as virtual links then they are equivalent as twisted links.  
Thus we have two maps, 
$$ 
\Phi: \{ \text{\rm classical links} \} \to \{ \text{\rm virtual links} \} \quad \text{\rm and} \quad 
\Psi: \{ \text{\rm virtual links} \} \to \{ \text{\rm twisted links} \}. 
$$
It is known that $\Phi$ is injective (\cite{rGPV}) and so is $\Psi \circ \Phi: \{ \text{\rm classical links} \} \to \{ \text{\rm twisted links} \}$ (\cite{rbor, rkk2020}). Thus, virtual links and twisted links are generalizations of classical links.  
The map $\Psi$ is not injective. However, it is understood well when two virtual links are equivalent as twisted links (\cite{rkk2020}).

In \cite{rkk16}, a method of constructing a virtual link diagram called a {\it double covering diagram} from a given twisted link diagram is introduced.

\begin{thm}[\cite{rkk16}]\label{thm:doublecover}
Let $D$ and $D'$ be twisted link diagrams, and $\widetilde{D}$ and $\widetilde{D'}$  their double covering diagrams, respectively. 
If $D$ and $D'$ represent the same twisted link, then $\widetilde{D}$ and $\widetilde{D'}$ represent the same virtual link. 
\end{thm} 

The {\it double cover\/} or the {\it double covering virtual link\/} of a twisted link $L$, denoted by $\widetilde{L}$, is defined as 
the virtual link represented by a double covering diagram of a diagram of $L$. 

Using the theorem, for a virtual link invariant $f$ we can obtain a twisted link invariant $\tilde{f}$ by 
$$ \tilde{f}(D) = f( \widetilde{D} ) \quad \text{\rm or} \quad \tilde{f}(L) = f( \widetilde{L} )$$ 
for a twisted link diagram $D$ or a twisted link $L$, where $\widetilde{D}$ is a double covering diagram of $D$ and $\widetilde{L}$ is the double covering virtual link of $L$.

\begin{definition}\label{def:doubledinvlink}{\rm 
For a twisted link diagram $D$, we define 
$\widetilde{J}^{\ast}_n(D)$, $\widetilde{P}^{\ast}_D(t)$ 
$( \ast \in \{ {\rm O}, {\rm U} \} )$ 
 and  $\widetilde{P}^{\rm OU}_L(s, t)$ to be 
$J^{\ast}_n(\widetilde{D})$, $P^{\ast}_{\widetilde{D}}(t)$ and $P^{\rm OU}_{\widetilde{D}}(s, t)$, respectively, where $\widetilde{D}$ is a double covering diagram of $D$. 
}\end{definition} 

By Theorem~\ref{thm:doublecover}, these are invariants of twisted links. 

When $D$ is a twisted knot diagram, by construction, the double covering diagram $\widetilde{D}$ is a virtual knot diagram or a $2$-component virtual link diagram.  We say that $D$ is of {\it odd type} (or {\it even type}) if  $\widetilde{D}$ is a virtual knot diagram (or a $2$-component virtual link diagram). A twisted knot diagram is of odd type (or even type) if and only if the number of bars on $D$ is odd (or even).  

For a virtual knot invariant $f$, we can obtain an invariant $\tilde{f}$ of a twisted knot of odd type by 
$$ \tilde{f}(D) = f( \widetilde{D} ) \quad \text{\rm or} \quad \tilde{f}(K) = f( \widetilde{K} )$$ 
for a twisted knot diagram $D$ or a twisted knot $K$ of odd type, where $\widetilde{D}$ is a double covering diagram of $D$ and $\widetilde{K}$ is the double covering virtual knot of $K$.   

\begin{definition}\label{def:doubledinvknot}{\rm 
For a twisted knot diagram $D$ of odd type, we define 
$\widetilde{J}_n(D)$ $(n \neq 0)$ and $\widetilde{P}_D(t)$ 
to be 
$J_n(\widetilde{D})$ and $P_{\widetilde{D}}(t)$, respectively, where $\widetilde{D}$ is a double covering diagram of $D$. 
}\end{definition} 

By Theorem~\ref{thm:doublecover}, these are invariants of twisted knots of odd type.

\section{An invariant of twisted knots and links}\label{sect:newtwistedinv}

In this section, we introduce a new twisted link invariant using over and under affine indices 
${\rm ind}^{\rm O}(c)$ and ${\rm ind}^{\rm U}(c)$ modulo $2$. 

Let $D$ be a twisted link diagram and 
${\rm Self}X(D)$ be the set of classical self-crossings of $D$. 
Let $c \in {\rm Self}X(D)$ and let $D_c$ be the twisted link diagram obtained from $D$ by smoothing at $c$. 
The over counting component $D_c^{\rm O}$ and the under counting component $D_c^{\rm U}$ for $D$ at $c$ are defined as in 
Figure~\ref{fig:smoothOUpositive} or \ref{fig:smoothOUnegative}.  

For $ \ast \in \{ {\rm O}, {\rm U} \} $, 
let $\bar{\rho}^{\ast}(c)$ be $1$ (or  $0$) if ${\rm ind}^{\ast}(c)$ is odd (or even).  
Here we assume that ${\rm ind}^{\ast}(c)$ 
$( \ast \in \{ {\rm O}, {\rm U} \} )$ 
is defined as before by ignoring bars appearing on the counting component $D_c^{\ast}$. 

For $ \ast \in \{ {\rm O}, {\rm U} \} $, 
let $p^{\ast}(c)$ be $1$ (or  $0$) if the number of bars on $D_c^{\ast}$ is odd (or even). 

\begin{definition}\label{def:newtwistedinv}{\rm 
For a twisted link diagram $D$, we define a polynomial $Q_D(s,t)$ by 

$$Q_D(s,t)=\sum_{c \in {\rm Self}X(D)} {\rm sgn}(c)(s^{\bar{\rho}^{\rm O}(c)}t^{p^{\rm O}(c)}-1)(s^{\bar{\rho}^{\rm U}(c)}t^{p^{\rm U}(c)}-1).$$
}
\end{definition}

\begin{thm}\label{thm:newtwistedinv}
The polynomial $Q_D(s,t)$ is an invariant of a twisted link.
\end{thm}

\begin{proof}

Virtual Reidemeister moves do not change $Q_D(s,t)$. Thus, it is sufficient to show that $Q_D(s,t)$ is preserved under Reidemeister moves and twisted moves. 

(i) Let $D'$ be obtained from $D$ by a Reidemesiter move I as in Figure~\ref{fgmoves}, where we assume 
$X(D') = X(D) \cup \{ c_1 \}$. 
For $c \in X(D) = X(D') \setminus \{ c_1 \}$, we see that 
$\bar{\rho}^{\ast}(c)$ and $p^{\ast}(c)$ $( \ast \in \{ {\rm O}, {\rm U} \} )$ do not change.  
Since 
$\bar{\rho}^{\ast}(c_1)=0$ and $p^{\ast}(c_1)=0$ for at least one $\ast \in \{ {\rm O}, {\rm U} \}$,  
we have $Q_D(s,t) = Q_{D'}D(s,t)$.  

(ii) Let $D'$ be obtained from $D$ by a Reidemesiter move II as in Figure~\ref{fgmoves}. 
For $c \in X(D) = X(D') \setminus \{ c_1, c_2 \}$, we see that  
$\bar{\rho}^{\ast}(c)$ and $p^{\ast}(c)$ $( \ast \in \{ {\rm O}, {\rm U} \} )$ do not change.  
Note that $\bar{\rho}^{\ast}(c_1) = \bar{\rho}^{\ast}(c_2)$ and $p^{\ast}(c_1) = p^{\ast}(c_2)$. 
Since ${\rm sgn}(c_1) = - {\rm sgn}(c_2)$, we see that the contribution of 
$c_1$ cancels with that of $c_2$.  

(iii) Let $D'$ be obtained from $D$ by a Reidemesiter move III as in Figure~\ref{fgmoves}. 
For $c \in X(D) \setminus \{ c_1, c_2, c_3 \} = X(D') \setminus \{ c'_1, c'_2, c'_3 \}$, we see that  
$\bar{\rho}^{\ast}(c)$ and $p^{\ast}(c)$ $( \ast \in \{ {\rm O}, {\rm U} \} )$ do not change.  
Note that, for $i \in \{1, 2, 3\}$, ${\rm sgn}(c_i) = {\rm sgn}(c'_i)$ and for $\ast \in \{ {\rm O}, {\rm U} \}$,  
$\bar{\rho}^{\ast}(c_i) = \bar{\rho}^{\ast}(c'_i)$ and $p^{\ast}(c_i) = p^{\ast}(c'_i)$. 
Thus we have $Q_D(s,t) = Q_{D'}D(s,t)$.

(iv) Let $D'$ be obtained from $D$ by a twisted move I or II.  
In this case it is obvious that $Q_D(s,t) = Q_{D'}D(s,t)$.  

(v) Let $D'$ be obtained from $D$ by a twisted move III as in Figure~\ref{fgtmoves}. 
For $c \in X(D) \setminus \{ c_1 \} = X(D') \setminus \{ c'_1 \}$, we see that  
$\bar{\rho}^{\ast}(c)$ and $p^{\ast}(c)$ $( \ast \in \{ {\rm O}, {\rm U} \} )$ do not change.  
Note that 
$D_{c_1}^{\rm O} = {D'}_{c'_1}^{\rm U}$ and 
$D_{c_1}^{\rm U} = {D'}_{c'_1}^{\rm O}$ 
outside of the region where the move III is applied. 
Then we see that 
$$
\bar{\rho}^{\rm O}(c_1) = \bar{\rho}^{\rm U}(c'_1),  \quad 
p^{\rm O}(c_1) = p^{\rm U}(c'_1)  \quad \text{\rm and} \quad 
\bar{\rho}^{\rm U}(c_1) = \bar{\rho}^{\rm O}(c'_1),  \quad  
p^{\rm U}(c_1) = p^{\rm O}(c'_1). 
$$
Since ${\rm sgn}(c_1) = {\rm sgn}(c'_1)$, we have $Q_D(s,t) = Q_{D'}D(s,t)$.  

Thus we see that $Q_D(s,t)$ is an invariant of a twisted link. 
\end{proof}

When $D$ is a twisted knot diagram, $\bar{\rho}^{\rm O}(c) = \bar{\rho}^{\rm U}(c)$. Hence we may compute one of 
$\bar{\rho}^{\rm O}(c)$ and $\bar{\rho}^{\rm U}(c)$ for each $c \in {\rm Self}X(D)$. 

\begin{example}{\rm 
For the twisted knot diagram $D_1$ depicted in Figure~\ref{fgtwisted} (i), we have
$Q_{D_1}(s,t)=(t-1)^2$.   It implies that $D_1$ is not equivalent to the unknot.
In \cite{rkn2013}, the first author introduced a polynomial invariant which is similar to $Q_D(s,t)$. But the invariant is not the same with ours in this paper.  The invariant defined in \cite{rkn2013} does not distinguish $D_1$ and the unknot.
}\end{example}

In \cite{rIK2023} three kinds of index polynomials $\psi$, $\psi_{+1}$ and $\psi_{-1}$ are defined for a twisted link. 
The invariants in \cite{rIK2023} can also distinguish $D_1$ and the unknot.

\begin{example}{\rm 
For the twisted knot diagrams $D_2$ and $D_3$ depicted in Figure~\ref{fgtwisted}, we have
$Q_{D_2}(s,t)=2(t-1)^2$ and $Q_{D_3}(s,t)=-2(st-1)^2$. So the invariant $Q_D(s,t)$ can distinguish $D_2$ and $D_3$ as twisted knots.
In \cite{rSaw}, J. Sawollek introduced a virtual link invariant, which is so-called the JKSS invariant,  from an invariant of links in thickened closed oriented surfaces defined by F. Jaeger, L. Kauffman, and H. Saleur \cite{rjks}.  In \cite{rkk16}, we introduced a method called the double covering, of sending twisted links to virtual links. 
Taking the JKSS invariant of the double covering, we have a twisted link invariant discussed in \cite{rkn20170}.  The values of this invariant of $D_2$ and $D_3$ are both zero. Hence they cannot be distinguished by the JKSS invariants of the double covers. 
}\end{example}



\end{document}